\documentclass[a4paper,12pt]{article}

\usepackage[intlimits]{amsmath}
\usepackage{amssymb}
\usepackage{amsthm}
\usepackage{mathrsfs}
\usepackage{cite}
\usepackage{bm}
\usepackage[colorlinks=true,linkcolor=blue,citecolor=blue,urlcolor=blue]{hyperref}

\newcommand{\la}{\langle}
\newcommand{\ra}{\rangle}
\newcommand{\pr}{\partial}
\newcommand{\dom}{\Omega}

\newcommand{\R}{\mathbb{R}}
\newcommand{\C}{\mathbb{C}}

\newcommand{\Ord}{\mathscr{O}}
\newcommand{\dd}{\,\mathrm{d}}

\newtheorem{thm}{Theorem}[section]
\newtheorem{prop}[thm]{Proposition}
\newtheorem{lem}[thm]{Lemma}
\newtheorem{cor}[thm]{Corollary}
\theoremstyle{definition}
\newtheorem{dfn}[thm]{Definition}
\theoremstyle{remark}
\newtheorem{rem}[thm]{Remark}

\title{Inverse boundary value problems for  certain doubly nonlinear parabolic and elliptic equations}
\author{%
C\u{a}t\u{a}lin I. C\^{a}rstea\thanks{
Department of Applied Mathematics,
National Yang Ming Chiao Tung University,
Hsinchu 300, Taiwan, R.O.C.,
\texttt{catalin.carstea@gmail.com}}
\and
Tuhin Ghosh\thanks{
Harish-Chandra Research Institute, Homi Bhabha National Institute,
Chhatnag Road, Jhunsi, Prayagraj (Allahabad) 211 019, India,
\texttt{tuhinghosh@hri.res.in}}}
\date{}

\begin{document}
\maketitle

\begin{abstract}
We consider an inverse boundary value problem for the doubly nonlinear parabolic equation
\[
\epsilon(x)\pr_t u^m-\nabla\cdot\bigl(\gamma(x)|\nabla u|^{p-2}\nabla u\bigr)=0
\quad\text{in }(0,T)\times\Omega,
\]
where $p\in(1,\infty)\setminus\{2\}$, $m>0$, and the coefficients $\epsilon$ and $\gamma$ are positive. Our first main result shows that when $m>p-1$, the lateral Cauchy data determine both coefficients. The proof proceeds by reducing the parabolic inverse problem to an inverse problem for the nonlinear elliptic equation
\[
-\nabla\cdot\bigl(\gamma|\nabla w|^{p-2}\nabla w\bigr)+Vw^m=0
\quad\text{in }\Omega.
\]
Our second main result establishes uniqueness for the pair $(\gamma,V)$ from the nonlinear Dirichlet-to-Neumann map of this elliptic equation. The argument has two steps. First, asymptotic expansions of the elliptic Dirichlet-to-Neumann map recover the weighted $p$-Laplacian Dirichlet-to-Neumann map, and and from it the coefficient $\gamma$. Second, once $\gamma$ is known, linearization at a noncritical background solution yields recovery of $V$. In dimension two we work under a simply connectedness assumption on the domain, while in dimensions $n\ge 3$ we assume that the conductivity is invariant in one known direction.
\end{abstract}

\section{Introduction}

Let $\Omega\subset\R^n$ be a smooth bounded domain, let $T>0$, and let $p\in(1,\infty)\setminus\{2\}$, $m>0$. In this paper we study the inverse problem for the doubly nonlinear parabolic equation
\begin{equation}\label{eq:parabolic}
\epsilon(x)\pr_t u^m(t,x)-\nabla\cdot\bigl(\gamma(x)|\nabla u|^{p-2}\nabla u\bigr)(t,x)=0
\qquad \text{in }Q_T:=(0,T)\times\Omega,
\end{equation}
with zero initial data and prescribed lateral Dirichlet data on $S_T:=(0,T)\times\pr\Omega$. Throughout we assume that
\begin{equation}\label{eq:coeff-bounds}
0<\mu^{-1}\le \epsilon(x),\gamma(x)\le \mu<\infty,
\qquad x\in\Omega,
\end{equation}
for some $\mu>1$, and, for the sake of simplicity, that $\epsilon,\gamma\in C^\infty(\Omega)$. 

Equations of the form \eqref{eq:parabolic} belong to the class of doubly nonlinear parabolic equations. This family includes, as special cases or limiting regimes, the porous medium equation and the parabolic $p$-Laplace equation, but it also displays phenomena that are specific to the simultaneous nonlinearity in the time derivative and the diffusion flux. From the analytical point of view, such equations provide a natural setting for degenerate and singular diffusion, intrinsic-scaling regularity theory, smoothing effects, and long-time asymptotic analysis; see, for instance, \cite{PorVes,BoDuMaSche-ARMA,BonGri,StanVaz}. From the modeling point of view, doubly nonlinear parabolic equations arise in nonlinear filtration for non-Newtonian fluids \cite{OtSu}, in phase-transition models written in enthalpy variables \cite{Visintin}, in glacier and ice-sheet dynamics \cite{GreBl}, and in reaction--diffusion problems with doubly nonlinear diffusion \cite{AudVaz}. 

We formulate the parabolic inverse problem in terms of the \emph{lateral Cauchy data set}
\begin{equation}\label{eq:intro-cauchy}
\mathcal C^{\mathrm{lat}}_{\epsilon,\gamma}
:=\left\{\left(f,\,\gamma|\nabla u|^{p-2}\partial_\nu u\big|_{S_T}\right):
\begin{array}{l}
 u \text{ is a weak solution of \eqref{eq:parabolic} in }Q_T,\\
 u(0,\cdot)=0,\ \ u\big|_{S_T}=f
\end{array}
\right\}.
\end{equation}
 We do not need, for the purposes of this paper, a general forward theory for arbitrary lateral Dirichlet data. Instead, the proof will use a distinguished class of variable separated boundary traces, for which existence is obtained later from an associated elliptic problem, while uniqueness follows from a comparison principle; see Sections~\ref{sec:parabolic} and \ref{sec:reduction}.

The key observation is that when $m>p-1$, the variable separaton Ansatz
\[
u(t,x)=t^\alpha w(x),\qquad \alpha=\frac{1}{m-p+1},\]
reduces \eqref{eq:parabolic} to the nonlinear elliptic equation
\begin{equation}\label{eq:elliptic-intro}
-\nabla\cdot\bigl(\gamma|\nabla w|^{p-2}\nabla w\bigr)+Vw^m=0
\qquad \text{in }\Omega,
\end{equation}
with
\begin{equation}\label{eq:V-from-epsilon}
V=\frac{m}{m-p+1}\,\epsilon.
\end{equation}
For \eqref{eq:elliptic-intro}, and for each strictly positive boundary value $g$, we write $w_g$ for the corresponding weak solution and define the associated nonlinear Dirichlet-to-Neumann map by
\begin{equation}\label{eq:intro-dn}
\langle \Lambda_{\gamma,V}(g),h\rangle
:=\int_\Omega \gamma|\nabla w_g|^{p-2}\nabla w_g\cdot\nabla \widetilde h\,dx
+\int_\Omega Vw_g^m\widetilde h\,dx,
\end{equation}
for $h\in C^\infty(\pr\Omega)$ and any extension $\widetilde h\in C^\infty(\overline\Omega)$ with $\widetilde h|_{\pr\Omega}=h$. The forward elliptic results needed to justify \eqref{eq:intro-dn} are established in Section~\ref{sec:elliptic-forward}.

Inverse problems for nonlinear equations have been studied extensively in recent years. For semilinear elliptic and parabolic equations, see for example \cite{Is1,IsSy,IsNa,FeOk,KrUh,KrUh2,LaLiiLinSa1,LaLiiLinSa2,Sun2}. For quasilinear elliptic equations and related nonlinear conductivity problems, we refer to  \cite{HeSun,Sun1,SunUh,Sun3,KaNa,EgPiSc,CarNaVa,CarFeKiKrUh,MuUh,Sh,CarFe1,CarFe2, Car2}. For inverse problems involving the $p$-Laplacian, the weighted $p$-Laplacian, and degenerate or singular nonlinear diffusion equations, see \cite{Br,BrKaSa,BrHaKaSa,BrIlKa,GuKaSa,KaWa,SaZh,CarKa,CarGhNa,CarGhUh,CarFe3, CarZim}. Most of these works use second- or higher-order linearization, a method introduced for semilinear parabolic equations in \cite{Is1}. The techniques we use in this paper are related to this method, but more closely resemble asymptotic techniques developed in \cite{CarKa, CarFe3, Car2, CarZim}.

Our first main result is the following uniqueness theorem for the parabolic lateral Cauchy data set.

\begin{thm}\label{thm:main-parabolic}
Let $m>p-1$, and let $\epsilon,\gamma,\tilde\epsilon,\tilde\gamma\in C^\infty(\Omega)$ be strictly positive coefficients satisfying \eqref{eq:coeff-bounds}. Assume that
\[
\mathcal C^{\mathrm{lat}}_{\epsilon,\gamma}=\mathcal C^{\mathrm{lat}}_{\tilde\epsilon,\tilde\gamma}.
\]
Then the following assertions hold.
\begin{enumerate}
\item If $n=2$ and $\Omega$ is simply connected, then
\[
\epsilon=\tilde\epsilon,\qquad \gamma=\tilde\gamma\qquad \text{in }\Omega.
\]
\item If $n\ge 3$ and there exists a nonzero vector $\xi\in\R^n$ such that
\[
\xi\cdot\nabla\gamma=\xi\cdot\nabla\tilde\gamma=0\qquad \text{in }\Omega,
\]
then
\[
\epsilon=\tilde\epsilon,\qquad \gamma=\tilde\gamma\qquad \text{in }\Omega.
\]
\end{enumerate}
\end{thm}

The proof of Theorem~\ref{thm:main-parabolic} proceeds by reducing the parabolic lateral Cauchy data to the nonlinear Dirichlet-to-Neumann map of \eqref{eq:elliptic-intro}. This leads to our second main result, which is the principal analytic statement of the paper.

\begin{thm}\label{thm:main-elliptic}
Let $m\neq p-1$, and let $\gamma,V,\tilde\gamma,\tilde V\in C^\infty(\Omega)$ be strictly positive. Suppose that
\[
\Lambda_{\gamma,V}=\Lambda_{\tilde\gamma,\tilde V}
\]
for the nonlinear Dirichlet-to-Neumann maps associated with
\begin{equation}\label{eq:elliptic-main}
\left\{
\begin{aligned}
-\nabla\cdot\bigl(\gamma|\nabla w|^{p-2}\nabla w\bigr)+Vw^m&=0 &&\text{in }\Omega,\\
 w&=g &&\text{on }\pr\Omega,
\end{aligned}
\right.
\end{equation}
where $\gamma,V,\tilde\gamma,\tilde V\in C^\infty(\Omega)$ are strictly positive. Then the following assertions hold.
\begin{enumerate}
\item If $n=2$ and $\Omega$ is simply connected, then
\[
\gamma=\tilde\gamma,\qquad V=\tilde V\qquad \text{in }\Omega.
\]
\item If $n\ge 3$ and there exists a nonzero vector $\xi\in\R^n$ such that
\[
\xi\cdot\nabla\gamma=\xi\cdot\nabla\tilde\gamma=0\qquad \text{in }\Omega,
\]
then
\[
\gamma=\tilde\gamma,\qquad V=\tilde V\qquad \text{in }\Omega.
\]
\end{enumerate}
\end{thm}

The proof of Theorem~\ref{thm:main-elliptic} has two main steps. First, we derive asymptotic expansions for the nonlinear map $\Lambda_{\gamma,V}$ in the small-data regime when $m>p-1$ and in the large-data regime when $m<p-1$. In both cases, the leading term is the Dirichlet-to-Neumann map of the weighted $p$-Laplacian, which yields recovery of $\gamma$ through the  results of \cite{CarFe3}. Second, once $\gamma$ is known, we linearize \eqref{eq:elliptic-main} at a noncritical background solution and recover $V$ from the induced linear Dirichlet-to-Neumann map. In dimension two this step uses the anisotropic Calder\'on problem result  of \cite{CarLiiTzo}; in higher dimensions it relies on a complex geometrical optics argument after reducing the common invariant direction to the $x_n$-direction by an orthogonal change of variables.

The structure of the paper is as follows. In Section~\ref{sec:parabolic} we define the lateral Cauchy data set and establish a comparison principle for the parabolic equation. In Section~\ref{sec:reduction} we show that separated solutions reduce the parabolic problem to the nonlinear elliptic equation \eqref{eq:elliptic-main}. Section~\ref{sec:elliptic-forward} contains the forward elliptic facts needed to define $\Lambda_{\gamma,V}$. In Section~\ref{sec:asymptotics} we derive the asymptotic expansions that recover the weighted $p$-Laplacian Dirichlet-to-Neumann map. Section~\ref{sec:inverse-elliptic} proves Theorem~\ref{thm:main-elliptic}. Finally, Section~\ref{sec:inverse-parabolic} deduces Theorem~\ref{thm:main-parabolic}.

\section{The parabolic setting and lateral Cauchy data}\label{sec:parabolic}

We consider nonnegative weak solutions of \eqref{eq:parabolic} with zero initial data.

\begin{dfn}\label{def:weak-solution}
Let $\phi\in C(\overline{Q_T};\R_+)$ satisfy $\phi(0,\cdot)=0$ on $\pr\Omega$. A function
\[
u\in L^p\bigl(0,T;W^{1,p}(\Omega)\bigr)\cap C\bigl([0,T];L^{m+1}(\Omega)\bigr),\qquad u\ge 0,
\]
is called a weak solution of \eqref{eq:parabolic} with lateral boundary value $\phi$ and zero initial value if
\begin{enumerate}
\item for every
\[
\varphi\in W_0^{1,m+1}(0,T;L^{m+1}(\Omega))\cap L^p(0,T;W^{1,p}_0(\Omega))
\]
we have
\begin{equation}\label{eq:weak-parabolic}
\int_{Q_T}\gamma|\nabla u|^{p-2}\nabla u\cdot\nabla\varphi\dd t\dd x-
\int_{Q_T}\epsilon u^m\,\pr_t\varphi\dd t\dd x=0;
\end{equation}
\item $u(0,\cdot)=0$ in $L^{m+1}(\Omega)$;
\item $u-\phi\in L^p\bigl(0,T;W_0^{1,p}(\Omega)\bigr)$.
\end{enumerate}
\end{dfn}

\begin{dfn}\label{def:lateral-cauchy}
The lateral Cauchy data set $\mathcal C^{\mathrm{lat}}_{\epsilon,\gamma}$ is the collection of pairs
\[
\bigl(\phi,\gamma|\nabla u|^{p-2}\pr_\nu u|_{S_T}\bigr)
\]
for which $u$ is a weak solution of \eqref{eq:parabolic} in the sense of Definition~\ref{def:weak-solution}.
\end{dfn}

For $0<h<T$ and $f\in L^1(Q_T)$, the Steklov average is defined by
\[
[f]_h(t,x)=\frac{\mathbf{1}_{(0,T-h)}(t)}{h}\int_t^{t+h}f(\tau,x)\dd\tau.
\]
We use the following known lemmas.

\begin{lem}[Chapter I, Lemma 3.2 in \cite{DiB}]\label{lem:dibenedetto}
If $f\in L^r(0,T;L^q(\Omega))$, then for every $\varepsilon>0$ one has $[f]_h\to f$ in $L^r(0,T-\varepsilon;L^q(\Omega))$ as $h\to0^+$. If $f\in C(0,T;L^q(\Omega))$, then $[f]_h(t,\cdot)\to f(t,\cdot)$ in $L^q(\Omega)$ for every $t\in(0,T-\varepsilon)$.
\end{lem}

For $\delta>0$, let
\[
H_\delta(z)=
\begin{cases}
1,& z\ge \delta,\\
\frac{z}{\delta},&0<z<\delta,\\
0,& z\le 0,
\end{cases}
\qquad
G_\delta(z)=
\begin{cases}
z-\frac{\delta}{2},& z\ge \delta,\\
\frac{z^2}{2\delta},&0<z<\delta,\\
0,& z\le 0.
\end{cases}
\]
Then $G_\delta'=H_\delta$.

\begin{lem}[Lemma 3.1 in \cite{BoStru}]\label{lem:Steklov-positive}
If $0<h<T$ and $f\in C(0,T;L^1(\Omega))$, then
\[
\pr_t[G_\delta(f)]_h\le \pr_t[f]_h\,H_\delta(f)
\qquad\text{a.e. in }Q_T.
\]
\end{lem}

The next proposition is a variant of the comparison principle in \cite{BoStru} adapted to \eqref{eq:parabolic}.

\begin{prop}\label{prop:comparison}
Let $u_1$ and $u_2$ be weak solutions of
\[
\epsilon(x)\pr_t u_j^m-\nabla\cdot\bigl(\gamma(x)|\nabla u_j|^{p-2}\nabla u_j\bigr)=f_j,
\qquad j=1,2.
\]
Assume that there exist $0\le t_1<t_2\le T$ such that
\begin{enumerate}
\item $u_1\le u_2$ on $(t_1,t_2)\times\pr\Omega$;
\item $f_1\le f_2$ on $(t_1,t_2)\times\Omega$;
\item for one $j\in\{1,2\}$ there exists $\lambda>0$ such that $u_j\ge \lambda$ on $(t_1,t_2)\times\Omega$;
\item if $m>1$, then for one $j\in\{1,2\}$ there exists $M>0$ such that $u_j\le M$ on $(t_1,t_2)\times\Omega$.
\end{enumerate}
Then
\begin{equation}\label{eq:comparison-ineq}
\int_\Omega \epsilon(x)\bigl(u_1^m(t_2,x)-u_2^m(t_2,x)\bigr)_+\dd x
\le
\int_\Omega \epsilon(x)\bigl(u_1^m(t_1,x)-u_2^m(t_1,x)\bigr)_+\dd x.
\end{equation}
\end{prop}

\begin{proof}
Let $h\in(0,T)$ and $t\in(0,T-h)$. In \eqref{eq:weak-parabolic} we may use\footnote{After an appropriate approximation and passage to the limit.}
\[
\varphi(\tau,x)=\frac{1}{h}\mathbf{1}_{(t,t+h)}(\tau)\psi(x),
\qquad \psi\in W^{1,p}_0(\Omega),
\]
to obtain
\[
\int_\Omega\Bigl(\epsilon\,\partial_t[u_1^m]_h\,\psi+\gamma[|\nabla u_1|^{p-2}\nabla u_1]_h\cdot\nabla\psi\Bigr)\,dx
=\int_\Omega [f_1]_h\psi\,dx,
\]
and the corresponding identity for $u_2$. Subtracting the two identities yields, for a.e. $t\in(0,T-h)$,
\begin{multline}\label{eq:comparison-difference}
\int_\Omega\epsilon\,\partial_t[u_1^m-u_2^m]_h\,\psi\,dx\\
=\int_\Omega\Bigl(\gamma\bigl[|\nabla u_2|^{p-2}\nabla u_2-|\nabla u_1|^{p-2}\nabla u_1\bigr]_h\cdot\nabla\psi+[f_1-f_2]_h\psi\Bigr)\,dx.
\end{multline}

Fix
\[
0<\delta<\min\!\left(1,\frac{\lambda^m}{2^m}\right)
\]
and set
\[
\psi=\mathbf{1}_{(t_1,t_2)}(t)H_\delta(u_1^m-u_2^m).
\]
Let
\[
D_\delta=\{(t,x)\in Q_T:t_1<t<t_2,\ 0<u_1^m-u_2^m<\delta\}.
\]
Then
\[
\nabla H_\delta(u_1^m-u_2^m)=\frac{1}{\delta}\mathbf{1}_{D_\delta}\nabla(u_1^m-u_2^m).
\]
Depending on the value of $m$, there are two cases to consider.

If $0<m\le 1$, either $u_2\ge \lambda$ on $(t_1,t_2)\times\Omega$, in which case $u_1\ge \lambda$ on $D_\delta$, or $u_1\ge \lambda$ on $(t_1,t_2)\times\Omega$, in which case $u_2\ge \lambda/2$ on $D_\delta$. In either case,
\[
|\nabla u_j^m|=m u_j^{m-1}|\nabla u_j|\le m\left(\frac{2}{\lambda}\right)^{1-m}|\nabla u_j|
\qquad\text{on }D_\delta,
\]
for $j\in\{1,2\}$.

If $m>1$, then by assumption there exists $M>0$ such that $u_j\le M$ on $(t_1,t_2)\times\Omega$ for one $j\in\{1,2\}$. Since $\delta<1$, it follows that both $u_1$ and $u_2$ satisfy
\[u_j<(1+M^m)^{1/m}
\qquad\text{on }D_\delta.
\]
Hence
\[
|\nabla u_j^m|\le m(1+M^m)^{\frac{m-1}{m}}|\nabla u_j|
\qquad\text{on }D_\delta,
\]
for $j\in\{1,2\}$.

Since $u_1\le u_2$ on $(t_1,t_2)\times\partial\Omega$, it is immediate that
\[
\psi\in L^p(0,T;W^{1,p}_0(\Omega)),
\]
so this is an admissible test function for a.e. $t$. Integrating \eqref{eq:comparison-difference} over $(t_1,t_2)$ gives
\begin{multline}
\int_{(t_1,t_2)\times\Omega}\epsilon\,\partial_t[u_1^m-u_2^m]_hH_\delta(u_1^m-u_2^m)\,dx\,dt\\
\le \int_{(t_1,t_2)\times\Omega}\Bigl(\gamma\bigl[|\nabla u_2|^{p-2}\nabla u_2-|\nabla u_1|^{p-2}\nabla u_1\bigr]_h\cdot\nabla H_\delta(u_1^m-u_2^m)\\
+[f_1-f_2]_hH_\delta(u_1^m-u_2^m)\Bigr)\,dx\,dt,
\end{multline}
which, by Lemma~\ref{lem:Steklov-positive}, implies that
\begin{multline}\label{eq:comparison-difference-2}
\int_{(t_1,t_2)\times\Omega}\epsilon\,\partial_t[G_\delta(u_1^m-u_2^m)]_h\,dx\,dt\\
\le \int_{(t_1,t_2)\times\Omega}\Bigl(\gamma\bigl[|\nabla u_2|^{p-2}\nabla u_2-|\nabla u_1|^{p-2}\nabla u_1\bigr]_h\cdot\nabla H_\delta(u_1^m-u_2^m)\\
+[f_1-f_2]_hH_\delta(u_1^m-u_2^m)\Bigr)\,dx\,dt.
\end{multline}

Taking the limit $h\to0^+$ on the left-hand side of \eqref{eq:comparison-difference-2}, we obtain
\begin{multline}
\lim_{h\to0^+}\int_{(t_1,t_2)\times\Omega}\epsilon\,\partial_t[G_\delta(u_1^m-u_2^m)]_h\,dx\,dt\\
=\lim_{h\to0^+}\left(\int_\Omega \epsilon [G_\delta(u_1^m-u_2^m)]_h(t_2,x)\,dx-\int_\Omega \epsilon [G_\delta(u_1^m-u_2^m)]_h(t_1,x)\,dx\right)\\
=\int_\Omega \epsilon G_\delta(u_1^m(t_2)-u_2^m(t_2))\,dx-\int_\Omega \epsilon G_\delta(u_1^m(t_1)-u_2^m(t_1))\,dx.
\end{multline}
Hence, by monotone convergence,
\begin{multline}\label{eq:comparison-left-limit}
\lim_{\delta\to0^+}\lim_{h\to0^+}\int_{(t_1,t_2)\times\Omega}\epsilon\,\partial_t[G_\delta(u_1^m-u_2^m)]_h\,dx\,dt\\
=\int_\Omega \epsilon (u_1^m(t_2)-u_2^m(t_2))_+\,dx-\int_\Omega \epsilon (u_1^m(t_1)-u_2^m(t_1))_+\,dx.
\end{multline}

Since $f_1\le f_2$, we have
\[
\int_{(t_1,t_2)\times\Omega}[f_1-f_2]_hH_\delta(u_1^m-u_2^m)\,dx\,dt\le 0.
\]
For the gradient term we compute
\begin{multline}
\int_{(t_1,t_2)\times\Omega}\gamma\bigl[|\nabla u_2|^{p-2}\nabla u_2-|\nabla u_1|^{p-2}\nabla u_1\bigr]_h\cdot\nabla H_\delta(u_1^m-u_2^m)\,dx\,dt\\
=\frac{1}{\delta}\int_{D_\delta}\gamma\bigl[|\nabla u_2|^{p-2}\nabla u_2-|\nabla u_1|^{p-2}\nabla u_1\bigr]_h\cdot\nabla(u_1^m-u_2^m)\,dx\,dt.
\end{multline}
Since
\[
|\nabla u_2|^{p-2}\nabla u_2-|\nabla u_1|^{p-2}\nabla u_1\in L^{\frac{p}{p-1}}(Q_T),
\]
Lemma~\ref{lem:dibenedetto} yields
\[
\bigl[|\nabla u_2|^{p-2}\nabla u_2-|\nabla u_1|^{p-2}\nabla u_1\bigr]_h
\to |\nabla u_2|^{p-2}\nabla u_2-|\nabla u_1|^{p-2}\nabla u_1
\]
in $L^{\frac{p}{p-1}}((t_1,t_2)\times\Omega)$ as $h\to0^+$. As $\nabla(u_1^m-u_2^m)\in L^p(D_\delta)$, it follows that
\begin{multline}
\lim_{h\to0^+}\int_{(t_1,t_2)\times\Omega}\gamma\bigl[|\nabla u_2|^{p-2}\nabla u_2-|\nabla u_1|^{p-2}\nabla u_1\bigr]_h\cdot\nabla H_\delta(u_1^m-u_2^m)\,dx\,dt\\
=\frac{1}{\delta}\int_{D_\delta}\gamma\bigl(|\nabla u_2|^{p-2}\nabla u_2-|\nabla u_1|^{p-2}\nabla u_1\bigr)\cdot\nabla(u_1^m-u_2^m)\,dx\,dt.
\end{multline}
Now
\[
\nabla(u_1^m-u_2^m)=m u_1^{m-1}\nabla(u_1-u_2)+m(u_1^{m-1}-u_2^{m-1})\nabla u_2.
\]
By the monotonicity inequality
\[
\bigl(|\xi|^{p-2}\xi-|\zeta|^{p-2}\zeta\bigr)\cdot(\xi-\zeta)>0,
\qquad \xi,\zeta\in\mathbb R^n,\ \xi\neq\zeta,
\]
we obtain
\[
\int_{D_\delta}\gamma m u_1^{m-1}\bigl(|\nabla u_2|^{p-2}\nabla u_2-|\nabla u_1|^{p-2}\nabla u_1\bigr)\cdot\nabla(u_1-u_2)\,dx\,dt\le 0.
\]
On the other hand, on $D_\delta$ we have
\begin{multline}
|u_1^{m-1}-u_2^{m-1}|=\frac{m}{|m-1|}\left|\int_{u_2^m}^{u_1^m}s^{-1/m}\,ds\right|\\
\le \frac{m}{|m-1|}u_2^{-1}(u_1^m-u_2^m)\le \frac{m}{|m-1|}u_2^{-1}\delta.
\end{multline}
As observed above, either $u_2\ge \lambda$ on $(t_1,t_2)\times\Omega$, or $u_1\ge \lambda$ on $(t_1,t_2)\times\Omega$, in which case $u_2\ge \lambda/2$ on $D_\delta$. Therefore
\[
|u_1^{m-1}-u_2^{m-1}|\le C(m,\lambda)\delta
\qquad\text{on }D_\delta.
\]
Consequently,
\[
\lim_{\delta\to0^+}\frac{1}{\delta}\int_{D_\delta}\gamma m (u_1^{m-1}-u_2^{m-1}) \bigl(|\nabla u_2|^{p-2}\nabla u_2-|\nabla u_1|^{p-2}\nabla u_1\bigr)\cdot\nabla u_2\,dx\,dt=0,
\]
since $\bigcap_{0<\delta<\min(1,\lambda^m/2^m)}D_\delta=\varnothing$. Combining the preceding estimates with \eqref{eq:comparison-left-limit} gives \eqref{eq:comparison-ineq}.
\end{proof}

\begin{cor}\label{cor:comparison-unique}
Let $u_1$ and $u_2$ be weak solutions of \eqref{eq:parabolic} with the same zero initial data and the same strictly positive lateral Dirichlet data on $S_T$. If one of the two solutions is bounded away from zero on compact time intervals and, when $m>1$, one of them is bounded above on compact time intervals, then $u_1=u_2$ in $Q_T$.
\end{cor}

\begin{proof}
Apply Proposition~\ref{prop:comparison} first to $(u_1,u_2)$ and then to $(u_2,u_1)$ with $f_1=f_2=0$, and let $t_1\to0^+$.
\end{proof}

\section{Reduction to a nonlinear elliptic problem}\label{sec:reduction}

Assume from now on that $f\equiv 0$ in \eqref{eq:parabolic}. Let
\begin{equation}\label{eq:alpha}
\alpha=\frac{1}{m-p+1},
\end{equation}
which is well defined when $m>p-1$. For each strictly positive $g\in C^\infty(\pr\Omega)$, we write
\begin{equation}\label{eq:separated-trace}
\phi_g(t,x)=t^\alpha g(x), \qquad (t,x)\in S_T.
\end{equation}
Thus the reduction argument below produces solutions for the distinguished class of lateral boundary traces $\{\phi_g:g>0\}$ used in the proof of Theorem~\ref{thm:main-parabolic}.

\begin{prop}\label{prop:separated}
Assume that $m>p-1$ and let $g\in C^\infty(\pr\Omega)$ be strictly positive. If $w$ solves
\begin{equation}\label{eq:elliptic-reduction}
\left\{
\begin{aligned}
-\nabla\cdot\bigl(\gamma|\nabla w|^{p-2}\nabla w\bigr)+\alpha m\,\epsilon\,w^m&=0 &&\text{in }\Omega,\\
 w&=g &&\text{on }\pr\Omega,
\end{aligned}
\right.
\end{equation}
then
\[
u(t,x)=t^\alpha w(x)
\]
is a weak solution of \eqref{eq:parabolic} with zero initial value and lateral boundary value $t^\alpha g$.
\end{prop}

\begin{proof}
A direct calculation gives
\[
\epsilon\pr_t u^m=\alpha m t^{\alpha m-1}\epsilon w^m,
\qquad
\nabla\cdot\bigl(\gamma|\nabla u|^{p-2}\nabla u\bigr)=t^{\alpha(p-1)}\nabla\cdot\bigl(\gamma|\nabla w|^{p-2}\nabla w\bigr).
\]
Since $\alpha m-1=\alpha(p-1)$ by \eqref{eq:alpha}, the pointwise equation follows from \eqref{eq:elliptic-reduction}. The weak formulation is immediate by integration in time.
\end{proof}

\begin{cor}\label{cor:separated-unique}
Assume that $m>p-1$ and $g\in C^\infty(\pr\Omega)$ is strictly positive. Then the separated solution provided by Proposition~\ref{prop:separated} is the unique weak solution of \eqref{eq:parabolic} with zero initial value and lateral Dirichlet data $\phi_g$.
\end{cor}

\begin{proof}
Let $u_1=t^\alpha w$ be the separated solution and let $u_2$ be any other weak solution with the same lateral data. Since $g>0$, the maximum principle for the elliptic equation implies that $w$ is strictly positive in $\Omega$, hence $u_1$ is bounded away from zero on every strip $[t_1,t_2]\times\Omega$ with $0<t_1<t_2\le T$. If $m>1$, then $u_1$ is also bounded above on such strips. Applying Corollary~\ref{cor:comparison-unique} on $[t_1,t_2]$ and letting $t_1\to0^+$ gives $u_1=u_2$. Therefore the forward parabolic problem is well posed for the boundary trace $\phi_g$.
\end{proof}

\begin{cor}\label{cor:parabolic-to-elliptic}
Assume that $m>p-1$ and that
\[
\mathcal C^{\mathrm{lat}}_{\epsilon,\gamma}=\mathcal C^{\mathrm{lat}}_{\tilde\epsilon,\tilde\gamma}.
\]
Set
\[
V=\frac{m}{m-p+1}\epsilon,
\qquad
\tilde V=\frac{m}{m-p+1}\tilde\epsilon.
\]
Then the nonlinear Dirichlet-to-Neumann maps associated with \eqref{eq:elliptic-main} satisfy
\[
\Lambda_{\gamma,V}=\Lambda_{\tilde\gamma,\tilde V}.
\]
\end{cor}

\begin{proof}
For each strictly positive $g\in C^\infty(\pr\Omega)$, Proposition~\ref{prop:separated} and Corollary~\ref{cor:separated-unique} show that the parabolic solution with boundary trace $\phi_g$ is uniquely determined by the elliptic solution of \eqref{eq:elliptic-main}. Moreover, if $u(t,x)=t^\alpha w(x)$, then the associated lateral flux is
\[
\gamma|\nabla u|^{p-2}\pr_\nu u=t^{\alpha(p-1)}\gamma|\nabla w|^{p-2}\pr_\nu w
\qquad\text{on }S_T.
\]
Since $t^{\alpha(p-1)}$ is a fixed nonvanishing factor on $(0,T)$, equality of the lateral Cauchy data sets implies equality of the elliptic boundary fluxes for every strictly positive boundary value $g$, and hence equality of the Dirichlet-to-Neumann maps.
\end{proof}

\section{The nonlinear elliptic problem and its Dirichlet-to-Neumann map}\label{sec:elliptic-forward}

We now consider the nonlinear elliptic boundary value problem
\begin{equation}\label{eq:elliptic}
\left\{
\begin{aligned}
-\nabla\cdot\bigl(\gamma|\nabla w|^{p-2}\nabla w\bigr)+Vw^m&=0 &&\text{in }\Omega,\\
w&=g &&\text{on }\pr\Omega,
\end{aligned}
\right.
\end{equation}
where $V\in C^\infty(\Omega)$ is strictly positive.

\begin{dfn}\label{def:elliptic-weak}
Let $g\in W^{1,p}(\Omega)\cap L^{m+1}(\Omega)$ be nonnegative. A function
\[
w\in g+W^{1,p}_0(\Omega),\qquad w\ge0,
\]
is called a weak solution of \eqref{eq:elliptic} if
\begin{equation}\label{eq:weak-elliptic}
\int_\Omega \gamma|\nabla w|^{p-2}\nabla w\cdot\nabla\varphi\,dx
+\int_\Omega Vw^m\varphi\,dx=0
\end{equation}
for every $\varphi\in W^{1,p}_0(\Omega)\cap L^{m+1}(\Omega)$.
\end{dfn}

\begin{prop}\label{prop:elliptic-forward}
Let $g\in W^{1,p}(\Omega)\cap L^{m+1}(\Omega)$ be nonnegative. Then \eqref{eq:elliptic} admits a unique weak solution
\[
w\in g+W^{1,p}_0(\Omega),\qquad w\ge0.
\]
Moreover, the following assertions hold.
\begin{enumerate}
\item The solution $w$ is the unique minimizer of the energy functional
\begin{equation}\label{eq:energy}
\mathcal E[v]
=\int_\Omega\left(\frac{1}{p}\gamma|\nabla v|^p+\frac{1}{m+1}Vv^{m+1}\right)\,dx
\end{equation}
over the affine space $g+W^{1,p}_0(\Omega)$.
\item There exists a constant $C$, depending only on $\Omega$, $p$, $m$, and upper and lower bounds for $\gamma$ and $V$, such that
\begin{equation}\label{eq:elliptic-energy-bound}
\|w\|_{W^{1,p}(\Omega)}+\|w\|_{L^{m+1}(\Omega)}
\le C\bigl(1+\|g\|_{W^{1,p}(\Omega)}+\|g\|_{L^{m+1}(\Omega)}\bigr).
\end{equation}
\item If $g\in L^\infty(\pr\Omega)$, then
\begin{equation}\label{eq:elliptic-max-principle}
\|w\|_{L^\infty(\Omega)}\le \|g\|_{L^\infty(\pr\Omega)}.
\end{equation}
\item If $g\in C^{1,\beta}(\pr\Omega)$ for some $\beta\in(0,1)$, then $w\in C^{1,\beta}(\overline\Omega)$ and
\begin{equation}\label{eq:lieberman-estimate}
\|w\|_{C^{1,\beta}(\overline\Omega)}
\le C\bigl(1+\|g\|_{C^{1,\beta}(\pr\Omega)}\bigr).
\end{equation}
\end{enumerate}
\end{prop}

\begin{proof}
Existence follows by the direct method of the calculus of variations. Indeed, the functional \eqref{eq:energy} is coercive and weakly lower semicontinuous on $g+W^{1,p}_0(\Omega)$, so it has a minimizer. Since the integrand is strictly convex in $\nabla v$ and $v^{m+1}$ is strictly convex on $[0,\infty)$, the minimizer is unique. The Euler--Lagrange equation is precisely \eqref{eq:weak-elliptic}, so the minimizer is the unique weak solution.

The bound \eqref{eq:elliptic-energy-bound} follows by comparing $\mathcal E[w]$ with $\mathcal E[g]$ and using the positivity of $\gamma$ and $V$. The maximum principle \eqref{eq:elliptic-max-principle} follows from \cite[Theorem 2.1]{LeoPe}. Finally, the global $C^{1,\beta}$ estimate \eqref{eq:lieberman-estimate} follows from \cite[Theorem 1]{Lieberman1988}.
\end{proof}

\begin{dfn}\label{def:DN}
Let $g\in C^\infty(\pr\Omega)$ be strictly positive, and let $w_g$ denote the corresponding weak solution of \eqref{eq:elliptic}. The nonlinear Dirichlet-to-Neumann map
\[
\Lambda_{\gamma,V}:C^\infty(\pr\Omega;\R_+)\to \mathcal D'(\pr\Omega)
\]
is defined in weak form by
\begin{equation}\label{eq:duality-DN}
\la \Lambda_{\gamma,V}(g),h\ra
=
\int_\Omega \gamma|\nabla w_g|^{p-2}\nabla w_g\cdot\nabla\widetilde h\,dx
+\int_\Omega Vw_g^m\widetilde h\,dx,
\end{equation}
for every $h\in C^\infty(\pr\Omega)$ and every extension $\widetilde h\in C^\infty(\overline\Omega)$ satisfying $\widetilde h|_{\pr\Omega}=h$.
\end{dfn}

\begin{rem}
The right-hand side of \eqref{eq:duality-DN} is independent of the choice of extension $\widetilde h$. Indeed, if $\widetilde h_1-\widetilde h_2\in C^\infty_0(\Omega)$, then \eqref{eq:weak-elliptic} with $\varphi=\widetilde h_1-\widetilde h_2$ shows that the two expressions agree. When $w_g$ is smooth up to the boundary, \eqref{eq:duality-DN} coincides with the classical boundary flux
\[
\Lambda_{\gamma,V}(g)=\gamma|\nabla w_g|^{p-2}\partial_\nu w_g\big|_{\pr\Omega}.
\]
\end{rem}

\section{Asymptotics of the elliptic Dirichlet-to-Neumann map}\label{sec:asymptotics}

We write
\begin{equation}\label{eq:p-laplace}
\nabla\cdot\bigl(\gamma|\nabla v|^{p-2}\nabla v\bigr)=0
\qquad \text{in }\Omega,
\end{equation}
for the weighted $p$-Laplace equation. In this section we distinguish the two cases $m>p-1$, and $m<p-1$. In the first two cases we assume that $v\in C^\infty(\overline\Omega)$ is a solution of \eqref{eq:p-laplace} with no critical points in $\overline\Omega$. The smoothness assumption is consistent with the nonvanishing of $\nabla v$, by elliptic regularity theory.

For $\xi\in\R^n\setminus\{0\}$ define
\[
J_j(\xi)=|\xi|^{p-2}\xi_j,
\qquad j=1,\dots,n.
\]
Taylor's formula gives
\begin{equation}\label{eq:Taylor-J}
J_j(\zeta)=J_j(\xi)+\sum_{k=1}^n(\zeta_k-\xi_k)\int_0^1 \partial_{\xi_k}J_j\bigl(\xi+t(\zeta-\xi)\bigr)\,\dd t,
\end{equation}
where
\begin{equation}\label{eq:J-derivative}
\pr_{\xi_k}J_j(\xi)=|\xi|^{p-2}\left(\delta_{jk}+(p-2)\frac{\xi_j\xi_k}{|\xi|^2}\right).
\end{equation}
For a noncritical $v$, we write
\begin{equation}\label{eq:A-v}
A[v]=\bigl(A[v]_{jk}\bigr)_{j,k=1}^n,
\qquad
A[v]_{jk}=\gamma|\nabla v|^{p-2}\left(\delta_{jk}+(p-2)\frac{\pr_jv\,\pr_kv}{|\nabla v|^2}\right).
\end{equation}

\subsection{The case $m>p-1$}

Let $\lambda\in(0,1)$, and let $w_\lambda$ be the solution of
\begin{equation}\label{eq:w-lambda-small}
\left\{
\begin{aligned}
-\nabla\cdot\bigl(\gamma|\nabla w_\lambda|^{p-2}\nabla w_\lambda\bigr)+Vw_\lambda^m&=0 &&\text{in }\Omega,\\
 w_\lambda&=\lambda v &&\text{on }\pr\Omega.
\end{aligned}
\right.
\end{equation}
We make the Ansatz
\begin{equation}\label{eq:ansatz-small}
w_\lambda=\lambda\bigl(v+\lambda^{m-p+1}R_\lambda\bigr).
\end{equation}

\begin{prop}\label{prop:small-data}
Assume that $m>p-1$. Then the family $R_\lambda$ is bounded in $C^{1,\beta}(\overline\Omega)$, and there exists $R\in C^{1,\beta}(\overline\Omega)$ such that $R_\lambda\to R$ in $C^1(\overline\Omega)$ as $\lambda\to0^+$. Moreover, $R$ solves
\begin{equation}\label{eq:R-equation}
\left\{
\begin{aligned}
\nabla\cdot\bigl(A[v]\nabla R\bigr)&=Vv^m &&\text{in }\Omega,\\
R&=0 &&\text{on }\pr\Omega,
\end{aligned}
\right.
\end{equation}
and for every $\omega\in C^\infty(\overline\Omega)$,
\begin{multline}\label{eq:expansion-small}
\la\omega,\Lambda_{\gamma,V}(\lambda v|_{\pr\Omega})\ra
=\lambda^{p-1}\int_\Omega \gamma|\nabla v|^{p-2}\nabla v\cdot\nabla\omega\,\dd x\\
+\lambda^m\int_\Omega\Bigl(\nabla\omega\cdot A[v]\nabla R+V\omega v^m\Bigr)\,\dd x+o(\lambda^m).
\end{multline}
\end{prop}

\begin{proof}
By Proposition~\ref{prop:elliptic-forward}, the family
\[
v+\lambda^{m-p+1}R_\lambda=\lambda^{-1}w_\lambda
\]
is bounded in $C^{1,\beta}(\overline\Omega)$ uniformly in $\lambda\in(0,1)$. By the Arzel\`a--Ascoli theorem, after passing to a subsequence there exists $\tilde v\in C^{1,\beta}(\overline\Omega)$ such that $\tilde v|_{\pr\Omega}=v|_{\pr\Omega}$ and
\[
v+\lambda^{m-p+1}R_\lambda\to \tilde v
\qquad\text{in }C^1(\overline\Omega)
\]
as $\lambda\to0^+$.

We next identify the limit. Since $w_\lambda$ solves \eqref{eq:w-lambda-small}, we have
\[
\lambda^{1-p}\Bigl(-\nabla\cdot\bigl(\gamma|\nabla w_\lambda|^{p-2}\nabla w_\lambda\bigr)+Vw_\lambda^m\Bigr)=0
\]
in the sense of distributions. Using \eqref{eq:ansatz-small}, the convergence above, and the fact that $m>p-1$, we obtain
\[
\mathscr D'(\Omega)\text{-}\lim_{\lambda\to0^+}\lambda^{1-p}Vw_\lambda^m=0
\]
and
\[
\mathscr D'(\Omega)\text{-}\lim_{\lambda\to0^+}\lambda^{1-p}\nabla\cdot\bigl(\gamma|\nabla w_\lambda|^{p-2}\nabla w_\lambda\bigr)
=\nabla\cdot\bigl(\gamma|\nabla\tilde v|^{p-2}\nabla\tilde v\bigr).
\]
Therefore $\tilde v$ satisfies the same weighted $p$-Laplace equation \eqref{eq:p-laplace} as $v$, with the same Dirichlet data. By uniqueness for the weighted $p$-Laplace equation with boundary value $v|_{\pr\Omega}$, we conclude that $\tilde v=v$. Hence
\[
\lambda^{m-p+1}R_\lambda\to0
\qquad\text{in }C^1(\overline\Omega)
\]
as $\lambda\to0^+$. Since the limit is independent of the subsequence chosen, the convergence holds for the full family.

Applying Taylor's formula \eqref{eq:Taylor-J} with
\[
\xi=\nabla v,
\qquad
\zeta=\nabla v+\lambda^{m-p+1}\nabla R_\lambda,
\]
we obtain
\[
\gamma|\nabla w_\lambda|^{p-2}\nabla w_\lambda
=\lambda^{p-1}\gamma|\nabla v|^{p-2}\nabla v+\lambda^m A_\lambda\nabla R_\lambda,
\]
where
\[
(A_\lambda)_{jk}=\int_0^1\gamma\,\pr_{\xi_k}J_j\bigl(\nabla v+t\lambda^{m-p+1}\nabla R_\lambda\bigr)\,\dd t.
\]
Since $v$ solves \eqref{eq:p-laplace}, it follows that
\[
\nabla\cdot(A_\lambda\nabla R_\lambda)=V\bigl(v+\lambda^{m-p+1}R_\lambda\bigr)^m
\qquad\text{in }\Omega,
\]
with $R_\lambda|_{\pr\Omega}=0$. In particular, because $\lambda^{m-p+1}R_\lambda\to0$ in $C^1(\overline\Omega)$ and $\nabla v$ has no zeros in $\overline\Omega$, the coefficient matrices $A_\lambda$ are uniformly elliptic for all sufficiently small $\lambda$, with ellipticity constants independent of $\lambda$. Moreover,
\[
V\bigl(v+\lambda^{m-p+1}R_\lambda\bigr)^m\to Vv^m
\qquad\text{in }C(\overline\Omega).
\]
Schauder estimates therefore imply that $R_\lambda$ is bounded in $C^{1,\beta}(\overline\Omega)$ uniformly in $\lambda$. Passing again to a subsequence, we obtain $R\in C^{1,\beta}(\overline\Omega)$ such that $R_\lambda\to R$ in $C^1(\overline\Omega)$, and the limit solves \eqref{eq:R-equation}. By uniqueness for \eqref{eq:R-equation}, the whole family converges to $R$.

Finally, by the weak definition of the Dirichlet-to-Neumann map,
\[
\la\omega,\Lambda_{\gamma,V}(\lambda v|_{\pr\Omega})\ra
=\int_\Omega \gamma|\nabla w_\lambda|^{p-2}\nabla w_\lambda\cdot\nabla\omega\,\dd x
+\int_\Omega V\omega w_\lambda^m\,\dd x.
\]
Using the expansion above for the flux term, together with the convergences $A_\lambda\to A[v]$ and $R_\lambda\to R$ in $C(\overline\Omega)$, we obtain
\begin{multline*}
\int_\Omega \gamma|\nabla w_\lambda|^{p-2}\nabla w_\lambda\cdot\nabla\omega\,\dd x
=\lambda^{p-1}\int_\Omega \gamma|\nabla v|^{p-2}\nabla v\cdot\nabla\omega\,\dd x\\[5pt]
+\lambda^m\int_\Omega \nabla\omega\cdot A[v]\nabla R\,\dd x+o(\lambda^m).
\end{multline*}
Also,
\[
w_\lambda^m=\lambda^m\bigl(v+\lambda^{m-p+1}R_\lambda\bigr)^m
=\lambda^m v^m+o(\lambda^m)
\qquad\text{uniformly in }\overline\Omega.
\]
Combining these two expansions yields \eqref{eq:expansion-small}.
\end{proof}

\subsection{The case $m<p-1$}

Let $\lambda\in(0,1)$, and let $w_\lambda$ be the solution of
\begin{equation}\label{eq:w-lambda-large}
\left\{
\begin{aligned}
-\nabla\cdot\bigl(\gamma|\nabla w_\lambda|^{p-2}\nabla w_\lambda\bigr)+Vw_\lambda^m&=0 &&\text{in }\Omega,\\
 w_\lambda&=\lambda^{-1}v &&\text{on }\pr\Omega.
\end{aligned}
\right.
\end{equation}
We make the Ansatz
\begin{equation}\label{eq:ansatz-large}
w_\lambda=\lambda^{-1}\bigl(v+\lambda^{p-1-m}R_\lambda\bigr).
\end{equation}

\begin{prop}\label{prop:large-data}
Assume that $m<p-1$. Then the family $R_\lambda$ is bounded in $C^{1,\beta}(\overline\Omega)$, and there exists $R\in C^{1,\beta}(\overline\Omega)$ such that $R_\lambda\to R$ in $C^1(\overline\Omega)$ as $\lambda\to0^+$. The limit $R$ solves \eqref{eq:R-equation}, and for every $\omega\in C^\infty(\overline\Omega)$,
\begin{multline}\label{eq:expansion-large}
\la\omega,\Lambda_{\gamma,V}(\lambda^{-1}v|_{\pr\Omega})\ra
=\lambda^{1-p}\int_\Omega \gamma|\nabla v|^{p-2}\nabla v\cdot\nabla\omega\,\dd x\\
+\lambda^{-m}\int_\Omega\Bigl(\nabla\omega\cdot A[v]\nabla R+V\omega v^m\Bigr)\,\dd x+o(\lambda^{-m}).
\end{multline}
\end{prop}

\begin{proof}
Note that
\[
\lambda w_\lambda=v+\lambda^{p-1-m}R_\lambda
\]
solves
\begin{equation}\label{eq:lambda-w-lambda-large}
\left\{
\begin{aligned}
-\nabla\cdot\bigl(\gamma|\nabla(\lambda w_\lambda)|^{p-2}\nabla(\lambda w_\lambda)\bigr)+\lambda^{p-1-m}V(\lambda w_\lambda)^m&=0 &&\text{in }\Omega,\\
 \lambda w_\lambda&=v &&\text{on }\pr\Omega.
\end{aligned}
\right.
\end{equation}
Since the family of potentials $\lambda^{p-1-m}V$ is uniformly bounded in $\lambda$, Proposition~\ref{prop:elliptic-forward} implies that the family
\[
v+\lambda^{p-1-m}R_\lambda=\lambda w_\lambda
\]
is bounded in $C^{1,\beta}(\overline\Omega)$ uniformly in $\lambda$. By Arzel\`a--Ascoli, after passing to a subsequence there exists $\tilde v\in C^{1,\beta}(\overline\Omega)$ such that $\tilde v|_{\pr\Omega}=v|_{\pr\Omega}$ and
\[
v+\lambda^{p-1-m}R_\lambda\to\tilde v
\qquad\text{in }C^1(\overline\Omega)
\]
as $\lambda\to0^+$.

As in the previous case, passing to the limit in \eqref{eq:lambda-w-lambda-large} in the sense of distributions gives
\[
\nabla\cdot\bigl(\gamma|\nabla\tilde v|^{p-2}\nabla\tilde v\bigr)=0
\qquad\text{in }\Omega.
\]
Since $\tilde v$ has the same boundary values as $v$, uniqueness for the weighted $p$-Laplace equation yields $\tilde v=v$. Therefore
\[
\lambda^{p-1-m}R_\lambda\to0
\qquad\text{in }C^1(\overline\Omega),
\]
and again the limit is independent of the subsequence.

Applying Taylor's formula \eqref{eq:Taylor-J} as above, we find that
\[
\gamma|\nabla w_\lambda|^{p-2}\nabla w_\lambda
=\lambda^{1-p}\gamma|\nabla v|^{p-2}\nabla v+\lambda^{-m}\widetilde A_\lambda\nabla R_\lambda,
\]
where
\[
(\widetilde A_\lambda)_{jk}=\int_0^1\gamma\,\pr_{\xi_k}J_j\bigl(\nabla v+t\lambda^{p-1-m}\nabla R_\lambda\bigr)\,\dd t.
\]
It follows that
\[
\nabla\cdot(\widetilde A_\lambda\nabla R_\lambda)=V\bigl(v+\lambda^{p-1-m}R_\lambda\bigr)^m
\qquad\text{in }\Omega,
\]
with $R_\lambda|_{\pr\Omega}=0$. Since $\lambda^{p-1-m}R_\lambda\to0$ in $C^1(\overline\Omega)$ and $\nabla v$ has no zeros, the coefficient matrices $\widetilde A_\lambda$ are uniformly elliptic independently of $\lambda$ for all sufficiently small $\lambda$. Consequently $R_\lambda$ is bounded in $C^{1,\beta}(\overline\Omega)$ uniformly in $\lambda$, and passing to the limit yields $R\in C^{1,\beta}(\overline\Omega)$ with $R_\lambda\to R$ in $C^1(\overline\Omega)$ and
\[
\nabla\cdot\bigl(A[v]\nabla R\bigr)=Vv^m
\qquad\text{in }\Omega,
\qquad
R|_{\pr\Omega}=0.
\]
By uniqueness, the full family converges.

Finally, using the weak definition of $\Lambda_{\gamma,V}$ together with the flux expansion above and the uniform convergence
\[
w_\lambda^m=\lambda^{-m}\bigl(v+\lambda^{p-1-m}R_\lambda\bigr)^m
=\lambda^{-m}v^m+o(\lambda^{-m}),
\]
we obtain
\begin{multline*}
\la\omega,\Lambda_{\gamma,V}(\lambda^{-1}v|_{\pr\Omega})\ra
=\lambda^{1-p}\int_\Omega \gamma|\nabla v|^{p-2}\nabla v\cdot\nabla\omega\,\dd x\\
+\lambda^{-m}\int_\Omega\Bigl(\nabla\omega\cdot A[v]\nabla R+V\omega v^m\Bigr)\,\dd x+o(\lambda^{-m}),
\end{multline*}
which is exactly \eqref{eq:expansion-large}.
\end{proof}

\begin{rem}\label{rem:critical}
The borderline case $m=p-1$ is excluded from the present argument. In that regime the two terms in the equation scale in the same way, so the asymptotic expansions above no longer separate the conductivity from the lower-order term. Different methods would be required.
\end{rem}

\section{Inverse problems for the elliptic equation}\label{sec:inverse-elliptic}

\subsection{Recovery of the weighted $p$-Laplacian Dirichlet-to-Neumann map}

Let $\Lambda_\gamma$ denote the Dirichlet-to-Neumann map for the weighted $p$-Laplace equation \eqref{eq:p-laplace}. We recall the following result from \cite{CarFe3}.

\begin{thm}[\cite{CarFe3}]\label{thm:CarFe}
Let $\gamma,\tilde\gamma\in C^\infty(\Omega)$ be strictly positive.
\begin{enumerate}
\item If $n=2$ and $\Lambda_\gamma=\Lambda_{\tilde\gamma}$, then $\gamma=\tilde\gamma$.
\item If $n\ge 3$ and there exists a nonzero vector $\xi\in\R^n$ such that
\[
\xi\cdot\nabla\gamma=\xi\cdot\nabla\tilde\gamma=0
\qquad \text{in }\Omega,
\]
and $\Lambda_\gamma=\Lambda_{\tilde\gamma}$, then $\gamma=\tilde\gamma$.
\end{enumerate}
\end{thm}

The next proposition is the first step in the proof of Theorem~\ref{thm:main-elliptic}.

\begin{prop}\label{prop:recover-gamma}
Assume that $m\neq p-1$ and that
\[
\Lambda_{\gamma,V}=\Lambda_{\tilde\gamma,\tilde V}.
\]
Then the weighted $p$-Laplacian Dirichlet-to-Neumann maps agree on every smooth boundary value that generates noncritical solutions for both conductivities. Consequently,
\begin{enumerate}
\item if $n=2$, then $\gamma=\tilde\gamma$;
\item if $n\ge 3$ and there exists a nonzero vector $\xi\in\R^n$ such that $\xi\cdot\nabla\gamma=\xi\cdot\nabla\tilde\gamma=0$, then $\gamma=\tilde\gamma$.
\end{enumerate}
\end{prop}

\begin{proof}
Comparing the leading terms in either \eqref{eq:expansion-small} or \eqref{eq:expansion-large}, we obtain
\[
\int_\Omega\gamma|\nabla v|^{p-2}\nabla v\cdot\nabla\omega\dd x
=
\int_\Omega\tilde\gamma|\nabla \tilde v|^{p-2}\nabla \tilde v\cdot\nabla\omega\dd x,
\]
whenever $v$ and $\tilde v$ solve the weighted $p$-Laplace equations for $\gamma$ and $\tilde\gamma$ with the same boundary value and have no critical points. Thus the nonlinear Dirichlet-to-Neumann maps for the weighted $p$-Laplacian agree on the class of noncritical data. As observed in \cite{CarFe3}, the proofs there use only such data, so Theorem~\ref{thm:CarFe} applies.
\end{proof}
\begin{rem}\label{rem:63}
Before proceeding further, we want to comment briefly on the existence of noncritical solutions to \eqref{eq:p-laplace}. In the $n=2$ case, by \cite[Theorem 5.1]{AlSi}, there exists boundary data $f\in C^\infty(\pr\dom)$, independent of $\gamma$, so that the corresponding solution $v$ of \eqref{eq:p-laplace} has no critical points. In fact, this will happen for any Dirichlet data that oscillates only once, i.e. it has only one local minimum and one local maximum along the boundary. In the $n\geq 3$ case, under the assumptions of Theorem \ref{thm:CarFe}, part 2., the solution with Dirichlet data $f=\xi\cdot x|_{\pr\dom}$ is  $v=\xi\cdot x$, which is noncritical
\end{rem}

\subsection{Linearization at a noncritical solution}

Assume now that $w_0$ is a noncritical solution of \eqref{eq:elliptic}. In the present setting such solutions exist, since the families constructed in Section~\ref{sec:asymptotics} are noncritical for sufficiently small or large parameter. See also Remark \ref{rem:63}. Let $f\in C^\infty(\pr\Omega)$ and let $\tau\in(-\tau_0,\tau_0)$ be small. Denote by $w_\tau$ the solution of
\begin{equation}\label{eq:w-tau}
\left\{
\begin{aligned}
-\nabla\cdot\bigl(\gamma|\nabla w_\tau|^{p-2}\nabla w_\tau\bigr)+Vw_\tau^m&=0 &&\text{in }\Omega,\\
 w_\tau&=w_0+\tau f &&\text{on }\pr\Omega.
\end{aligned}
\right.
\end{equation}
We make the Ansatz
\[
w_\tau=w_0+\tau R_\tau.
\]

\begin{prop}\label{prop:linearization}
The family $R_\tau$ converges in $C^1(\overline\Omega)$ to the unique solution $\dot w$ of
\begin{equation}\label{eq:linearized}
\left\{
\begin{aligned}
-\nabla\cdot\bigl(A[w_0]\nabla \dot w\bigr)+mVw_0^{m-1}\dot w&=0 &&\text{in }\Omega,\\
\dot w&=f &&\text{on }\pr\Omega.
\end{aligned}
\right.
\end{equation}
Moreover, for every $\omega\in C^\infty(\overline\Omega)$,
\begin{multline}\label{eq:expansion-mid}
\la \omega,\Lambda_{\gamma,V}(w_0|_{\pr\Omega}+\tau f)\ra
=
\int_\Omega\gamma|\nabla w_0|^{p-2}\nabla w_0\cdot\nabla\omega\,\dd x+
\int_\Omega V\omega w_0^m\,\dd x\\
+\tau\int_\Omega\Bigl(\nabla\omega\cdot A[w_0]\nabla\dot w+mVw_0^{m-1}\omega\dot w\Bigr)\,\dd x+o(\tau).
\end{multline}
Hence $\Lambda_{\gamma,V}$ determines the Dirichlet-to-Neumann map of the linearized equation \eqref{eq:linearized}.
\end{prop}

\begin{proof}
Since $w_\tau=w_0+\tau f$ on $\pr\Omega$, the family $\tau R_\tau$ has uniformly bounded boundary values. By Proposition~\ref{prop:elliptic-forward} applied to \eqref{eq:w-tau}, the family $\tau R_\tau=w_\tau-w_0$ is uniformly bounded in $C^{1,\beta}(\overline\Omega)$. Reasoning as before we can quickly obtain that $\tau R_\tau\to 0$ in $C^1(\overline\Omega)$.

As in the proof of Proposition~\ref{prop:small-data}, the equation for $w_\tau$ may be rewritten as
\[
-\nabla\cdot\bigl(A_\tau\nabla R_\tau\bigr)+\mathcal V_\tau R_\tau=0
\qquad\text{in }\Omega,
\]
with boundary value $R_\tau|_{\pr\Omega}=f$, where
\[
(A_\tau)_{jk}=\gamma\int_0^1 \partial_{\xi_k}J_j\bigl(\nabla w_0+t\tau\nabla R_\tau\bigr)\,\dd t,
\qquad
\mathcal V_\tau=mV\int_0^1 (w_0+t\tau R_\tau)^{m-1}\,\dd t.
\]
Because $\nabla w_0$ has no zeros and $\tau R_\tau\to 0$ in $C^1(\overline\Omega)$, the matrices $A_\tau$ satisfy ellipticity bounds that are uniform for small $\tau$. Likewise, the coefficients $\mathcal V_\tau$ are uniformly bounded in $L^\infty(\Omega)$. An  application of standard Schauder estimates therefore yields a uniform $C^{1,\beta}(\overline\Omega)$ bound for $R_\tau$. Passing to a further subsequence, we may assume that $R_\tau\to \dot w$ in $C^1(\overline\Omega)$.

The convergences $A_\tau\to A[w_0]$ and $\mathcal V_\tau\to mVw_0^{m-1}$ in $C(\overline\Omega)$ then imply that $\dot w$ solves \eqref{eq:linearized}. Since the latter equation has a unique solution, the whole family $R_\tau$ converges to $\dot w$ in $C^1(\overline\Omega)$.

Finally, by the weak definition of the Dirichlet-to-Neumann map and the same Taylor expansion used above,
\begin{align*}
\la \omega,\Lambda_{\gamma,V}(w_0|_{\pr\Omega}+\tau f)\ra
&=\int_\Omega \gamma J(\nabla w_\tau)\cdot\nabla\omega\,\dd x+\int_\Omega V\omega w_\tau^m\,\dd x\\
&=\int_\Omega \gamma|\nabla w_0|^{p-2}\nabla w_0\cdot\nabla\omega\,\dd x+\int_\Omega V\omega w_0^m\,\dd x\\
&\quad+\tau\int_\Omega\Bigl(\nabla\omega\cdot A[w_0]\nabla\dot w+mVw_0^{m-1}\omega\dot w\Bigr)\,\dd x+o(\tau),
\end{align*}
which is exactly \eqref{eq:expansion-mid}. The coefficient of $\tau$ is the weak Dirichlet-to-Neumann map of \eqref{eq:linearized}.
\end{proof}

\subsection{Recovery of $V$ in the plane}

We now assume that $n=2$ and that $\Omega$ is simply connected. By Proposition~\ref{prop:recover-gamma}, we may suppose that $\gamma=\tilde\gamma$. We choose a noncritical solution $w_0$ for the common conductivity, and denote by $\tilde w_0$ the corresponding solution for the coefficient $\tilde V$ with the same boundary values.

\begin{prop}\label{prop:recover-V-2d}
Assume that $n=2$, that $\Omega$ is simply connected, and that
\[
\gamma=\tilde\gamma,
\qquad
\Lambda_{\gamma,V}=\Lambda_{\gamma,\tilde V}.
\]
Then $V=\tilde V$ in $\Omega$.
\end{prop}

\begin{proof}
By Proposition~\ref{prop:linearization}, the nonlinear map $\Lambda_{\gamma,V}$ determines the Dirichlet-to-Neumann map of the linearized equation
\[
-\nabla\cdot(A[w_0]\nabla \dot w)+mVw_0^{m-1}\dot w=0.
\]
Since
\[
\det A[w_0]=(p-1)\gamma^2|\nabla w_0|^{2(p-2)},
\]
we introduce the Riemannian metric
\[
g[w_0]=(\det A[w_0])^{1/2}A[w_0]^{-1}
=\sqrt{p-1}\left(I+\frac{2-p}{p-1}\frac{\nabla w_0\otimes \nabla w_0}{|\nabla w_0|^2}\right).
\]
If we now set
\[
v=(\det A[w_0])^{-1/4}\dot w,
\]
then $v$ solves the Schr\"odinger equation
\begin{equation}\label{eq:surf-schrod}
-\Delta_{g[w_0]}v+Q[w_0,V]v=0
\end{equation}
on the surface $(\Omega,g[w_0])$, where
\[
Q[w_0,V]=(\det A[w_0])^{-1/4}\Bigl[\Delta_{g[w_0]}(\det A[w_0])^{1/4}+mVw_0^{m-1}\Bigr].
\]
The same construction applied to $\tilde V$ gives an analogous equation involving $\tilde w_0$, $g[\tilde w_0]$, and $Q[\tilde w_0,\tilde V]$. Since $w_0|_{\pr\Omega}=\tilde w_0|_{\pr\Omega}$ and $\nabla w_0|_{\pr\Omega}=\nabla \tilde w_0|_{\pr\Omega}$ are known from the nonlinear boundary data, Proposition~\ref{prop:linearization} shows that the equality of nonlinear Dirichlet-to-Neumann maps yields equality of the Dirichlet-to-Neumann maps for the two Schr\"odinger equations. Therefore \cite[Theorem 1.1]{CarLiiTzo} gives a boundary-fixing diffeomorphism $\Psi:\Omega\to\Omega$ such that
\begin{equation}\label{eq:transformation-g}
g[\tilde w_0]=\Psi^*g[w_0]
\end{equation}
and
\[
Q[\tilde w_0,\tilde V]=Q[w_0,V]\circ\Psi.
\]

Taking determinants in \eqref{eq:transformation-g}, we obtain $\det D\Psi=1$. Both $g[w_0]^{-1}$ and $g[\tilde w_0]^{-1}$ are orthogonally equivalent to the constant matrix
\[
\frac{1}{\sqrt{p-1}}\begin{pmatrix}1&0\\0&p-1\end{pmatrix}.
\]
It follows that $D\Psi$ maps one orthogonal basis of $\R^2$ to another, and hence $D\Psi$ is orthogonal. Since $\Psi$ fixes the boundary pointwise, we conclude that $\Psi=\mathrm{Id}$. Thus
\[
g[w_0]=g[\tilde w_0]
\qquad\text{and hence}\qquad
\frac{\nabla \tilde w_0}{|\nabla \tilde w_0|}=\frac{\nabla w_0}{|\nabla w_0|}.
\]

Because $\Omega$ is simply connected, no level set of $w_0$ can be a closed curve contained strictly inside $\Omega$: such a curve would bound a topological disc, and the normalized gradient field $\nabla w_0/|\nabla w_0|$ would then have to vanish somewhere inside by the Brouwer fixed point theorem, contradicting the noncriticality of $w_0$. Hence every level set of $w_0$ meets the boundary. By rotating the normalized gradient field by an angle of $\pi/2$, we obtain a smooth tangent field for the level sets, and the identity above shows that $w_0$ and $\tilde w_0$ have the same level curves. Since these level curves start on the boundary, where $w_0$ and $\tilde w_0$ agree, it follows that $w_0=\tilde w_0$ in $\Omega$. Returning to \eqref{eq:elliptic-main}, we conclude that $V=\tilde V$.
\end{proof}

\subsection{Recovery of $V$ in higher dimensions}

Assume now that $n\ge 3$ and that $\gamma=\tilde\gamma$ is invariant in one direction. After an orthogonal change of variables, we may suppose that
\begin{equation}\label{eq:directional-invariance}
\pr_n\gamma=0.
\end{equation}

\begin{prop}\label{prop:recover-V-nd}
Assume that $n\ge 3$, that \eqref{eq:directional-invariance} holds, and that
\[
\gamma=\tilde\gamma,
\qquad
\Lambda_{\gamma,V}=\Lambda_{\gamma,\tilde V}.
\]
Then $V=\tilde V$ in $\Omega$.
\end{prop}

\begin{proof}
By Proposition~\ref{prop:recover-gamma}, we may indeed assume that the conductivity is common. Therefore, in the asymptotic expansions of Section~\ref{sec:asymptotics}, we may take the same noncritical weighted $p$-harmonic function $v$ for both coefficients. Let $R$ and $\tilde R$ be the corresponding correction terms and set
\[
S=R-\tilde R,
\qquad
q=V-\tilde V.
\]
Then subtracting the two equations for $R$ and $\tilde R$ gives
\begin{equation}\label{eq:S-equation}
\left\{
\begin{aligned}
\nabla\cdot\bigl(A[v]\nabla S\bigr)&=qv^m &&\text{in }\Omega,\\
S&=0 &&\text{on }\pr\Omega.
\end{aligned}
\right.
\end{equation}
Moreover, comparing the second terms in \eqref{eq:expansion-small} or \eqref{eq:expansion-large} yields
\begin{equation}\label{eq:integral-S}
\int_\Omega\Bigl(\nabla S\cdot A[v]\nabla\omega+q\omega v^m\Bigr)\,\dd x=0
\end{equation}
for every smooth $\omega$.

We now perturb the background solution $v$. Let $v_0$ be a noncritical weighted $p$-harmonic function, let $f\in C^\infty(\pr\Omega)$, and let
\[
v_\tau=v_0+\tau\dot v+o(\tau)
\]
be the linearization of the weighted $p$-Laplace equation, as in \cite{CarFe3}. Choosing $v=v_\tau$ in \eqref{eq:integral-S}, differentiating at $\tau=0$, and integrating by parts in the term involving $\dot S$ gives
\begin{multline}\label{eq:integral-identity}
0=\int_\Omega \nabla \dot S \cdot A[v_0]\nabla\omega+\nabla S \cdot \dot A[v_0,\dot v]\nabla\omega+mqv_0^{m-1}\omega \dot v\,\dd x\\
=\int_\Omega mqv_0^{m-1}\omega \dot v\,\dd x-\int_\Omega S\,\nabla \cdot \bigl(\dot A[v_0,\dot v]\nabla\omega\bigr)\,\dd x,
\end{multline}
where $\dot S=\left.\frac{\dd}{\dd\tau}\right|_{\tau=0}S$ and
\begin{multline*}
\dot A[v_0,\dot v]=(p-2)\gamma|\nabla v_0|^{p-4}\Big[(\nabla v_0\cdot\nabla \dot v)I\hspace{7em}\\
+(p-4)(\nabla v_0\cdot\nabla \dot v)\frac{\nabla v_0\otimes\nabla v_0}{|\nabla v_0|^2}+\nabla v_0\otimes\nabla \dot v+\nabla \dot v\otimes\nabla v_0\Big].
\end{multline*}

We now choose
\[
v_0(x)=x_n.
\]
Because of \eqref{eq:directional-invariance}, this function is weighted $p$-harmonic. We then have
\[
A[v_0]=\gamma\bigl(I+(p-2)e_n\otimes e_n\bigr),
\]
which is the isotropic conductivity $\gamma I$ stretched in the direction $e_n$. This allows us to use the classical Sylvester--Uhlmann construction \cite{SyUh} to produce complex geometrical optics solutions for the operator $\nabla\cdot(A[v_0]\nabla\cdot)$. More precisely, if $\zeta\in\C^n$ satisfies
\begin{equation}\label{eq:CGO-condition}
\zeta\cdot\bigl(I+(p-2)e_n\otimes e_n\bigr)\zeta=0,
\end{equation}
and $|\zeta|$ is large enough, then there exists a solution of the form
\[
\omega_\zeta=\gamma^{-1/2}e^{\zeta\cdot x}(1+r_\zeta),
\qquad
\|r_\zeta\|_{H^s(\Omega)}\le \frac{C}{|\zeta|},
\quad s>\frac n2.
\]

Let $\xi,\eta,\mu\in\R^n\setminus\{0\}$ and $t,s\in\R$ be such that $\eta\perp\xi,\mu,e_n$, $|\eta|=|\mu|=1$, and $\xi$, $\mu$, and $e_n$ are coplanar. Define
\[
\zeta_\pm=\pm s\mu+i(\xi\pm t\eta).
\]
A direct computation gives
\begin{multline*}
\zeta_\pm\cdot\bigl(I+(p-2)e_n\otimes e_n\bigr)\zeta_\pm
=s^2-|\xi|^2-t^2\pm2is\mu\cdot\xi\\
+(p-2)\Big[s^2(\mu\cdot e_n)^2-(\xi\cdot e_n)^2\pm2is(\mu\cdot e_n)(\xi\cdot e_n)\Big].
\end{multline*}
Hence \eqref{eq:CGO-condition} is satisfied provided that
\begin{equation}\label{eq:angle-condition}
\mu\cdot\xi+(p-2)(\mu\cdot e_n)(\xi\cdot e_n)=0
\end{equation}
and
\begin{equation}\label{eq:s-condition}
s^2\Bigl[1+(p-2)(\mu\cdot e_n)^2\Bigr]=t^2+|\xi|^2\Bigl[1+(p-2)\frac{(\xi\cdot e_n)^2}{|\xi|^2}\Bigr].
\end{equation}
The identity \eqref{eq:angle-condition} is an orthogonality condition with respect to the inner product defined by $I+(p-2)e_n\otimes e_n$. Given any $\xi\in\R^n\setminus\R e_n$, we may choose $\mu\in\mathrm{span}\{\xi,e_n\}$ so that \eqref{eq:angle-condition} holds, and then choose $\eta$ orthogonal to the plane of $\xi$, $\mu$, and $e_n$. Finally, \eqref{eq:s-condition} is satisfied by taking
\[
s=\Bigl[1+(p-2)(\mu\cdot e_n)^2\Bigr]^{-1/2}
\Bigl[t^2+|\xi|^2+(p-2)(\xi\cdot e_n)^2\Bigr]^{1/2},
\]
with $t$ as a large parameter.

We now use in \eqref{eq:integral-identity} the choices $\omega=\omega_{\zeta_+}$ and $\dot v=\omega_{\zeta_-}$. Keeping only the highest-order terms in $t$, we compute
\begin{multline*}
\nabla\cdot\Bigl(\dot A[x_n,\omega_{\zeta_-}]\nabla\omega_{\zeta_+}\Bigr)
=e^{2i\xi\cdot x}\,2i(p-2)\xi\cdot\Bigl[-s(e_n\cdot\mu)(s\mu+it\eta)\hspace{7em}\\
-(p-4)s^2(e_n\cdot\mu)^2e_n-(s^2-t^2)e_n-s(e_n\cdot\mu)(s\mu+it\eta)\Bigr]+\Ord(t)\\
=-e^{2i\xi\cdot x}\,2i(p-2)\Bigl[2s^2(e_n\cdot\mu)(\xi\cdot\mu)
+(e_n\cdot\xi)\bigl(s^2(1+(p-4)(e_n\cdot\mu)^2)-t^2\bigr)\Bigr]+\Ord(t)\\
=-e^{2i\xi\cdot x}\,2i(p-2)(e_n\cdot\xi)
\Bigl[s^2\bigl(1-p(e_n\cdot\mu)^2\bigr)-t^2\Bigr]+\Ord(t)\\
=e^{2i\xi\cdot x}\frac{4i(p-2)(p-1)(e_n\cdot\mu)^2}{1+(p-2)(e_n\cdot\mu)^2}t^2+\Ord(t),
\end{multline*}
where we used the identity $\xi\cdot\mu=-(p-2)(e_n\cdot\mu)(e_n\cdot\xi)$ coming from \eqref{eq:angle-condition}. On the other hand,
\[
\int_\Omega mqx_n^{m-1}\omega_{\zeta_-}\omega_{\zeta_+}\,\dd x=\Ord(1)
\qquad\text{as }t\to\infty.
\]
Therefore \eqref{eq:integral-identity} implies that $\widehat S(2\xi)=0$ for an open set of $\xi$. Since $S$ is compactly supported in $\Omega$, analytic continuation yields $S\equiv0$. Returning to \eqref{eq:S-equation}, we obtain $q=0$, that is, $V=\tilde V$.
\end{proof}

We may now complete the proof of of our result in the elliptic case.

\begin{proof}[Proof of Theorem~\ref{thm:main-elliptic}]
By Proposition~\ref{prop:recover-gamma}, the nonlinear Dirichlet-to-Neumann map determines the conductivity $\gamma$ under the assumptions of the theorem. Once $\gamma$ is known, Proposition~\ref{prop:recover-V-2d} yields recovery of $V$ in the planar case, and Proposition~\ref{prop:recover-V-nd} yields recovery of $V$ in dimensions $n\ge 3$ under the one-directional invariance assumption.
\end{proof}

\section{The parabolic inverse problem}\label{sec:inverse-parabolic}

We now deduce the parabolic uniqueness theorem from the elliptic one.

\begin{proof}[Proof of Theorem~\ref{thm:main-parabolic}]
Assume that
\[
\mathcal C^{\mathrm{lat}}_{\epsilon,\gamma}=\mathcal C^{\mathrm{lat}}_{\tilde\epsilon,\tilde\gamma}
\]
and let
\[
V=\frac{m}{m-p+1}\epsilon,
\qquad
\tilde V=\frac{m}{m-p+1}\tilde\epsilon.
\]
Since $m>p-1$, Corollary~\ref{cor:parabolic-to-elliptic} shows that
\[
\Lambda_{\gamma,V}=\Lambda_{\tilde\gamma,\tilde V}.
\]
We now apply Theorem~\ref{thm:main-elliptic}.

If $n=2$ and $\Omega$ is simply connected, then Theorem~\ref{thm:main-elliptic} yields
\[
\gamma=\tilde\gamma,
\qquad
V=\tilde V
\qquad \text{in }\Omega.
\]
Using \eqref{eq:V-from-epsilon}, we obtain
\[
\epsilon=\tilde\epsilon
\qquad \text{in }\Omega.
\]

If $n\ge 3$ and there exists a nonzero vector $\xi\in\R^n$ such that
\[
\xi\cdot\nabla\gamma=\xi\cdot\nabla\tilde\gamma=0
\qquad \text{in }\Omega,
\]
then Theorem~\ref{thm:main-elliptic} again gives
\[
\gamma=\tilde\gamma,
\qquad
V=\tilde V
\qquad \text{in }\Omega,
\]
and therefore \eqref{eq:V-from-epsilon} implies
\[
\epsilon=\tilde\epsilon
\qquad \text{in }\Omega.
\]
This proves both assertions of Theorem~\ref{thm:main-parabolic}.
\end{proof}

%

\section*{Acknowledgements}
C.C. was supported by NSTC grant  113-2115-M-A49-018-MY3. T.G. was supported by grant number NBHM(R.P)/R\&D II/9464.

\bibliographystyle{plain}
\bibliography{nonlinearity_updated_v11}

\end{document}